\documentclass{amsart}
\title{Closed 1-form on topological spaces, and Lusternik-Schnirelman type Category}
\author{Kazuyoshi Watanabe}
\date{}
\usepackage{amsthm,amsmath,amssymb,graphicx}
\usepackage{amscd,pifont}
\theoremstyle{definition}
\newtheorem{dfn}{Definition}
\newtheorem{lem}{Lemma}
\newtheorem{prop}{Proposition}
\newtheorem{thm}{Theorem}
\newtheorem{cor}{Corollary}
\newtheorem{clm}{claim}
\newtheorem*{remark}{{\rm Remark}}
\newtheorem*{example}{{\rm Example}}

\address{Tohoku University, Sendai~985-8578, Japan}
\email{kazuyoshi.watanabe.q5@dc.tohoku.ac.jp}

\begin{document}
\maketitle
This article deals with a continuous closed 1-form defined on a CW-complex. In particular, we show Lusternik-Schnirelmann type theory on continuous closed 1-forms which is related to gradient-like flows. M.Farber defined a continuous closed 1-form and a category with a respect to a cohomology class and constructed a Lusternik-Schnirelman Theory for smooth closed 1-forms. We denote about a definition of a zero-point of continuous closed 1-form and generalize this theory to continuous closed 1-forms. It shows a relation between the number of a zero-point of continous closed 1-form and the category with a respect to a cohomology class, through a dynamical system on a original space, homoclinic cycle.
\section{Introduction}
In 1981 S.P.Novikov initiated Morse theory closed 1-form which gave topological estimates of the numbers of zeros of closed 1-forms \cite{nov}. It was motivated by several important problems in mathematical physics. His fundamental idea was based on a plan to construct a chain complex called the Novikov complex, which uses a dynamics of the gradient flows on the abelian covering associated with a cohomology class. Dynamics of gradient flows appears traditionally in Morse theory, providing a bridge between the critical set of a function and topology. At present the Novikov theory of closed 1-forms is a rapidly developing area in topology which interacts with various mathematics, for example in symplectic homology, Floer theory, and so on. \par
In Novikov theory, there are two another invariants called Novikov betti number $b_{j}(\xi)$ and Novikov tortion number $q_{j}(\xi)$ with respect to a cohomology class $\xi \in H^1 (X;{\bf R})$. So fare Novikov inequality are known, which give some relation between three invariants for a morse closed 1-form $\omega$:
$$
c_{j} (\omega) \leq b_{j} (\xi) + q_{j} (\xi) + q_{j-1} (\xi),
$$
where $c_{j} (\omega)$ is the number of zeros of morse 1-form $\omega$ having index $j$.

In 2002 M.Farber introduced a new Lusternik-Schnirelmann type theory for closed 1-forms in \cite{far2}, \cite{far3}. This theory aims at finding relations between topology of the zero set of a closed 1-form and homotopical information, based on the cohomology class of the form. In classical Lusternik-Schnirelmann theory, the category of a smooth manifold gives a lower bound of the number of critical points of smooth functions on a manifold. On the other hand, Lusternik-Schnirelmann theory for closed 1-forms is different from classical one. In fact in any non-zero cohomology class there always exists a closed 1-form having at most one zero. Hence, the it gives information not only about the number of zeros but also about qualitative dynamical properties of smooth flows on manifolds. Farber estimated this category by a cup-length similarly to classical Lusternik-Schnirelmann theory. To define the category with respect to a cohomology class, Farber introduced a closed 1-form on  topological space, which we call continuous closed 1-form \cite{far2}, \cite{far3}.
    
In section2 we define a continuous closed 1-form. It is introduced by M.Farber and is a generalization of a smooth closed 1-form on a smooth manifold.
\begin{dfn}
A continuous closed 1-form $\omega $ on  X is defined as a collection  $\{ f _{U} \}_{U \in \mathcal{ U}} $ of continuous real-valued functions  $f_{U} : U \rightarrow {\bf R}$ , where  $\mathcal{U} = \{ U \}$is an open cover of $X$, such that for any pair $U,V \in \mathcal{U}$
\begin{eqnarray}
f_U | _{U \cap V} - f_V | _{U \cap V} : U \cap V \rightarrow {\bf R}
\end{eqnarray}
is locally constant function.
\end{dfn}
Any cohomology class can be represented by a continuous closed 1-form on some reasonable class of topological spaces, and an integral of continous 1-form along path can be defined. So, continuous 1-forms have same properties with smooth 1-forms.

In section3 and 4, we define a category with respect to a cohomology class, and proof a Lusternik-Schnirelman Theory about a continous closed 1-form on a finite CW-complex. A zero-point of a smooth closed 1-form $\omega$ is defined as a point $p\in M$ with $\omega _{p} =0$ for a smooth section $M \rightarrow T^{*} M$. But a continuous closed 1-form is not represented by such sections, so the zero point of a continuous closed 1-form is defined by using gradient-like flows, that is a point $p\in X$ is said to be zero point of $\omega$ if $p$ is a fixed point for any gradient-like flow of $\omega$. For using this definition we describe the Lusternik-Schnirelman theorem.
 \begin{thm}
 Let $X$ be a finite CW-complex, and let $\omega$ be a continuous closed 1-form on $X$, that represents a cohomology class $\xi \in H^{1} (X ; {\bf R})$ Let $\phi$ be a gradient-like flow of $\omega$ and $\operatorname{Fix} (\varphi^{t}) $ is finite set ${p_{1},...,p_{k}} $, where integer $k>0$. If $k < \operatorname{cat}(X, \xi)$, then the flow $\varphi$ has a homoclinic cycle.
 \end{thm}
It follows from the definition of zero-points we obtain the following conclusions:
\begin{cor}
(1) Let $X$ be a finite CW-complex, and let $\omega$ be a continuous closed 1-form, and $\xi =[ \omega ] \in H ^1 ( X; {\bf R})$.Suppose that $\omega$ has finite zero points $p_1,..,p_k$ . If $k < \operatorname{cat}(X, \xi)$, then any gradient-like flow that having fixed points only $p_1,..,p_k$ has a homoclinic cycle.\\
 (2) Let $\omega$ be a continuous closed 1-form on a finite CW-complex $X$, and $\omega$ represents $\xi \in H^{1} (X; {\bf R})$ . If $\omega$ admits a gradient-like flow with no homoclinic cycles, then $\omega$ has at least $\operatorname{cat} (X,\xi)$ distinct zeros.
 \end{cor}
This cororally is analogue to smooth Lusternik-Schnirelmann Thoery about a smooth closed 1-form on a smooth compact manifold. And smooth thoery deals with a smooth synamical system, but this theory can deal with not only smooth manifolds but also finite CW-complexes, so we can consider about the continuous dynamical systems on CW-complexes.\par
 
\section{Definition of continuous closed 1-form}
We assume $X$ is a CW-complex. Now we define a continuous closed 1-form, it is introduced by M.Farber in \cite{far1}, \cite{far3}.
\begin{dfn}
A continuous closed 1-form $\omega $ on  X is defined as a collection  $\{ f _{U} \}_{U \in \mathcal{U}} $ of continuous real-valued functions  $f_{U} : U \rightarrow {\bf R}$ , where  $\mathcal{U} = \{ U \}$is an open cover of $X$, such that for any pair $U,V \in \mathcal{U}$,
\begin{eqnarray}
f_U | _{U \cap V} - f_V | _{U \cap V} : U \cap V \rightarrow {\bf R}
\end{eqnarray}
is locally constant function.
\end{dfn}
Two 1-forms, $\{ f _{U} \}_{U \in \mathcal{U}} $ and $\{ g_{V}  \}_{V \in \mathcal{V}}$ where $\mathcal{V}  $ is another open cover of $X$, are equivalent, if $\{f_U , g_V \} _{U\in \mathcal{U} , V\in \mathcal{V}} $ is a continuous closed 1-form on $X$. We can define a pullback of continuous closed 1-form. Let $\phi : Y \rightarrow X$ be continuous map and let $\omega =\{ f_U \}_{U \in \mathcal{U}} $ be a continuous closed 1-form on $X$, where $\mathcal{U} $ is an open cover of $X$.\par
The pull back of the continuous closed 1-form $\omega $ can be defined $\phi^{*} \omega = \{f_{U} \circ \phi \} _{\phi ^{-1}(U) \in \phi ^{-1} \mathcal{U}}$, where $\phi ^{-1} \mathcal{U} =  \{\phi ^{-1} (U) \}_{U \in \mathcal{U}}$. Any continuous function $f:X \rightarrow {\bf R}$ defines a closed 1-form on $X$ with open cover $\mathcal{U}= \{X \}$. This closed 1-form is denoted by $df$.

Let $\omega =\{ f _{U} \}_{U \in \mathcal{U}}$ be a continuous closed 1-form on $X$, and let $\gamma : [0,1] \rightarrow X$ be a continuous path. Find a subdivision $t_{0} =0 < t_
{1} < ... < t_{N} =1$ of interval {[0,1]} such that for any $i$ the image $\gamma [t_{i},t_{i+1}]$ is contained in a single open set $U_{i} \in \mathcal{U}$. Then we define an integral of a continuous closed 1-form $\omega$ along a path $\gamma$, about
\begin{eqnarray}
\int_{\gamma} \omega = \sum ^{N-1}_{i=1} [f_{U_{i}}(\gamma (t_{i+1})) - f_{U_{i}} (\gamma (t_{i}))].
\end{eqnarray}
The integral does not depend on the choice of the subdivision and the open cover $\mathcal{U}$, so the integral is well-defined.
\begin{lem}
Let $\omega$ be a continuous closed 1-form on X. For any pair of continuous paths $\gamma$, $\gamma '$ on X with common beginning and common end points, provided that $\gamma$ and $\gamma '$ are homotopic relative to the boundary, it holds that
\begin{eqnarray}
\int _{\gamma} \omega = \int _{\gamma '} \omega
\end{eqnarray}
\end{lem}
\begin{proof}
It is sufficient to consider the condition of 2 open covers. This case is easy. 
\begin{figure}[htbp]
  \begin{center}
 \includegraphics[width=3cm, bb=0 0 300 300]{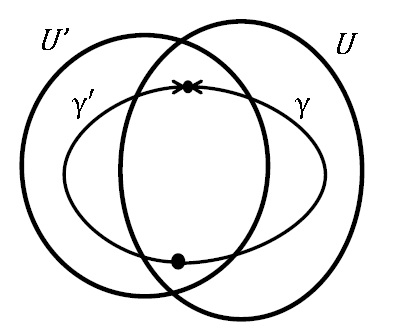}
\caption{The case that the open cover is $\{U,U' \}$}
\end{center}
\end{figure}
\end{proof}

By lemma, any continuous closed 1-form $\omega$ defines the homomorphism of periods
\begin{eqnarray}
\pi _{1} (X,x_{0}) \rightarrow {\bf R},\
[\gamma ] \mapsto \int _{\gamma } \omega \in {\bf R}.
\end{eqnarray}
The homomorphism of periods is group homomorphism.

\begin{lem}
Let $X$ be a locally path-connected topological space. A continuous closed 1-form $\omega$ on $X$ equals $df$ for a continuous function $f :X \rightarrow {\bf R}$ if and only if for any choice of the base point $x_{0} \in X$ the homomorphism of periods determined by  $\omega$ is trivial.
\end{lem}
\begin{proof} 
We denote the outline of proof.

If $\omega = df$, it is obvious.

Conversely, assume that the homomorphism of periods is trivial. In each connected component of $X$ choose a base point $x_{i} \in X$. One defines a continuous function $f:X \rightarrow {\bf R}$ by
\begin{eqnarray}
f(x) = \int _{x_{i}} ^{x} \omega
\end{eqnarray}
here $x$ and $x_{i}$ belong to the same component of $X$ and integration is taken along an arbitrary path connecting $x_{i}$ to $x$. Then $df = \omega$.
\end{proof}
\subsection{Representations of the cohomology class}

Any continuous closed 1-form $\omega$ on a topological space $X$ defines a singular cohomology class $[\omega] \in H^1 (X;{\bf R})$. It is defined by the homomorphism of periods viewed as an element of $\operatorname{Hom} (H_{1} (X;{\bf R})) = H^1 (X;{\bf R})$.
\begin{lem}
Let $X$ be a paracompact Hausdorff homologically locally 1-connected topological space. Then any singular cohomology class $\xi \in H^{1}(X;{\bf R} )$ may be realized by a continuous closed 1-from on $X$.
\end{lem}
\begin{proof}
We denote outline of the proof.\par
$ {\bf R}_{X} $ denotes the sheaf of locally constant function, $C_{X}$ denotes the sheaf of real value constant function, $ B_{X} $ denotes the sheaf of germ of constant function modulo locally constant function. The space of global sections of $B_{X}$, $H^{0} (X; B_{X})$ coincide with the space of continuous closed 1-form on $X$. We obtain the next sequence,
$$
\begin{CD}
0 @>>> H^{0}(X; {\bf R}) @>>> H^{0} (X; C_{X}) @>>> \\
 H^{0} (X; B_{X}) @>>> H^{1}(X; {\bf R}) @>[\: ]>> 0
\end{CD}
$$
The group $H^{1} (X;{\bf R})$ is the \'{C}ech cohomology $\check{H}^{1}(X;{\bf R})$ and the map [ ] assigns to a closed 1-from $\omega$ its cohomology class $[\omega ] \in \check{H}^{1} (X; {\bf R} )$. The natural map $\check{H}^{1} (X; {\bf R} ) \rightarrow H^{1} (X; {\bf R}) $ is an isomorphism. This implies our statement. \end{proof}\par

  As we have shown above, the space of continuous closed 1-forms $H^0 (X; B_{X})$ on a path-coccected topological space $X$ can be described by the following short exact sequence:
  
 \begin{eqnarray}
 0 \rightarrow C(X) / {\bf R} \rightarrow H^0(X; B_{X}) \rightarrow \check{H} ^1 (X; {\bf R}) \rightarrow 0.
 \end{eqnarray}

\begin{cor}
A closed 1-form $\omega$ on $X$ equals $df$ for some continuous function $f: X \rightarrow {\bf R}$ if and only if the \'{C}ech cohomology class $[\omega] \in \check{H} ^1 (X; {\bf R})$ vanishes, $[\omega] =0$
\end{cor}

\section{Category with respect to a cohomology class}
In this section we denote a definition of  a category with respect to a cohomology class in \cite{far1}. A classical Lusternik-Schnirelaman category is a minimal number of open covers, and this category is a generalization of a classical one.
\subsection{Definition of category with respect to a cohomology class}
\begin{dfn}
Let $X$ be a finite CW-complex and let $\xi \in H^{1}(X; {\bf R})$ be a real cohomology class. We define $\operatorname{cat}(X,\xi ) $ to be the least integer k such that for any integer $N>0$ there exists an open cover 
\begin{eqnarray}
X=F \cup F_{1} \cup \cdots \cup F_{k},
\end{eqnarray}
such that\\
(a) Each inclusion $F_{j} \rightarrow X$ is null-homotopic, where j=1,...,k\\
(b) There exist a homotopy $h_{t} : F \rightarrow X$, where $t \in [0,1]$, such that $h_{0}$ is the inclusion $F \rightarrow X$ and for any point $x \in F$,\par
\begin{eqnarray}
\int _{\gamma _{x}} \omega \leq -N
\end{eqnarray}
where the curve $\gamma _{x} :[0,1] \rightarrow X$ is given by $\gamma _{x} (t) = h_{t} (x)$ and $\omega$ is a continuous closed 1-form representing the cohomology class $\xi$.
\end{dfn}
 $\operatorname{cat}(X, \xi )$ does not depend on the choice of the continuous closed 1-form $\omega$. Indeed, if $\omega '$ is another continuous closed 1-form representing $\xi$, then $\omega - \omega ' = df$, where $f:X \rightarrow {\bf R}$ is a continuous function, and any continuous curve  $\gamma : [0,1] \rightarrow X$,
 \begin{eqnarray}
|\int_{\gamma} \omega - \int_{\gamma} \omega ' | =  |f(\gamma (1) ) - f(\gamma (0))| \leq C
 \end{eqnarray}
where the constant $C$ is independent of $\gamma$. Here we have used the compactness of $X$. This shows that if one can construct open covers such that (b) is satisfied with arbitrary large $N>0$ then the same is true with $\omega '$ replacing $\omega $.\\

\begin{remark}

 In general the inequality  
\begin{eqnarray}
\operatorname{cat}(X,\xi ) \leq \operatorname{cat}(X)
\end{eqnarray}
 holds. We may always consider covers with $F$ empty. Here $\operatorname{cat}(X)$ denotes the classical Lusternik-Schnierelman category of $X$.\par
 It cam be also proved that
 \begin{eqnarray}
 \operatorname{cat}(X,\xi ) \leq \operatorname{cat}(X)
 \end{eqnarray}
  for connected $X$ and nonzero $\xi \in H^1 (X ; {\bf R})$. This follows because for any $N>0$ we can take an open set $F$ that is contractible to a point in $X$ as the (b) of the definition. So if $X=G_{1}\cup \cdots G_r $ is a categorical open cover, we may construct an open cover that $F=G_1$ and $F_j = G_{j+1} $ for $j=1,...,r-1$, which satisfies the definition.\par
 Observe also that $\operatorname{cat}(X, \xi ) = \operatorname{cat}(X, \lambda \xi )$, for $\lambda \in {\bf R}, \lambda >0$. It is clearly from the definition.
 \end{remark}
 \subsection{Another definition}
 Sometimes it is more convinient to use a different version of condition (b) of Definition. Let $p: \tilde{X} \rightarrow X$ be the normal covering correspoding to the kernel of the homomorphism of periods. The  group $\Gamma$ of covering transformation of this covering equals the image of the homomorphism of periods. The induced closed 1-form $p^{*} \omega$ equals $df$ where $f: \tilde{X} \rightarrow {\bf R}$ is a continuous function. Denote $p^{-1} (F)$ by $F$, it is an open subset invariant under $\Gamma$.Condition(b) is equivalent to\par
 (b') There exists a homotopy $\tilde{h}_{t} : \tilde{F} \rightarrow \tilde{X}$, where $t \in [0,1]$, such that $\tilde{h}_{0}$ is the inclusion $\tilde{F} \rightarrow \tilde{X}$, each $\tilde{h}_{t}$ is $\Gamma$-equivalent, for any point $x\in \tilde{F}$,
 \begin{eqnarray}
 f(\tilde{h}_{1} (x) ) - f(x) \leq -N
 \end{eqnarray}
  The homotpy $\tilde{h}_{t}$ is a lift of the $h_{t}$, which exists because of the homotopy lifting property of coverings. If $\tilde{\gamma}_{y}$, where $y \in \tilde{F}$ and $p(y)=x$, denotes the path $\tilde{\gamma}_{y} (t) = \tilde{h}_{t} (y)$ in $\tilde{X}$, then $\gamma_ {x} = p_{*} \tilde{\gamma}_{y}$ and
 \begin{eqnarray}
 \int_{\gamma _{x} } \omega =\int _{\tilde{\gamma }_{y} } df = f(\tilde{h}_{1} (y)) - f(x)
 \end{eqnarray}
 This shows the equivalence (b) and (b').
 
 \subsection{Homotopy invariance}
\begin{lem}
 Let $X_{1}$ and $X_{2}$ be finite CW-complexes. Let $\phi : X_{1} \rightarrow X_{2}$ be a homotopy equivalence, $\xi_{2} \in H^{1} (X_{2} ;{\bf R})$, and $\xi _{1} = \phi ^{*} (\xi _{2}) \in H^{1} (X_{1} ; {\bf R})$. Then
\begin{eqnarray}
\operatorname{cat} (X_{1} , \xi _{1}) = \operatorname{cat} (X_{2}, \xi _{2}).
\end{eqnarray}
\end{lem}
 \begin{proof}
 We denote an outline of proof.\par
 Let $\psi : X_{2} \rightarrow X_{1}$ be a homotopy inverse of $\phi$. Choose a closed 1-form $\omega_{1}$ on $X_{1}$. Then $\omega_{2} = \psi ^{*} \omega _{1}$ is a closed 1-form on $X_{2}$ lying in the cohomology class $\xi _{2}$.\par
  Fix a homotopy $r_{t} : X_{1} \rightarrow X_{2}$, where $t \in [0,1]$, such that $r_{0} = \verb|id|_{X_{1} } $ and $r_{1} = \psi \circ \phi$. Compactness of $X_{1}$ implies that there is a constant $C>0$ such that $|\int_{\alpha _{x}} \omega _{1}| < C$ for any point $x\in X_{1}$, where $\alpha_{x} (t)=r_{t} (x) $, where $t\in [0,1]$.\par
  Suppose that $\operatorname{cat}(X_{2}, \xi _{2} ) \leq k$. Given any $N >0$, there is an open covering $X_{2} = F \cup F_{1} \cup \cdots \cup F_{k}$, such that $F_{1},...,F_{k}$ are contractible in $X_{2}$ and there exists a homotopy $h_{t} :F \rightarrow X_{2}$, where $t \in [0,1]$, such that $\int _{\gamma_{x}} \omega_{2} \leq -N-C$ for any $x \in F$, where $\gamma _{x} (t) = h_{t} (x)$. Define
\begin{eqnarray}
G= \phi ^{-1} (F), G_{j} = \phi ^{-1} (F_{j}),j=1,...,k
\end{eqnarray}
 These sets form an open cover of $X$, $X=G\cup G_{1} \cup \cdots \cup G_{k}$. This cover satisfy the definition of category. The argument proves that $\operatorname{cat} (X_{1} , \xi _{1})  \leq \operatorname{cat}(X_{2} , \xi _{2})$.\par
 The inverse inequality follows similarly.
 \end{proof}
\section{The Estimate of the number of zeros} 
\subsection{Zero point of a sooth closed 1-form}
For a classical Lusternik-Schnirelmann category, the category on a manifold gives a lower bound of any smooth function. But the Lusternik-Schnirelmann type theory for closed 1-forms is very deferent from both the classical Lusternik-Schnirelmann theory for smooth functions and from the Novikov theory. Because in any non-zero cohomology class $\xi \in H^{1} (X ;{\bf Z})$, one can always find a reperesenting closed 1-form having at of one zero[cf. \cite{far1}]. Hence quantitative estimates for the number of zeros can be obtained only under some dynamical properties of garadient-like vector fields for a given closed 1-form.\par
 Let $v$ be a smooth vector field on a closed smooth manifold $M$. Recall that a homoclinic orbit of $v$ is defined as an integral trajectory $\gamma (t)$,
$$
d\gamma (t) / dt = v(\gamma (t)), t\in {\bf R}
$$
such that the limits $\lim _{t \rightarrow + \infty} \gamma (t)$ and $\lim _{t \rightarrow - \infty} \gamma (t)$ both exist and are equal:
$$
\lim _{t \rightarrow + \infty} \gamma (t) =\lim _{t \rightarrow - \infty} \gamma (t)
$$
 More generally, a homoclinic cycle of length of n is a sequence of integral trajectories $\gamma _{1} (t), \gamma _{2} , ... ,\gamma _{n} (t)$ of $v$ such that,
 \begin{eqnarray}
 \lim_{t \rightarrow + \infty} \gamma _{i} (t) = p_{i+1} = \lim_{t \rightarrow - \infty } \gamma _{i+1} (t)
 \end{eqnarray}
for $i=1,...,n-1$ and
\begin{eqnarray}
 \lim_{t \rightarrow + \infty} \gamma _{n} (t) = p_{i+1} = \lim_{t \rightarrow - \infty } \gamma _{1} (t).
\end{eqnarray}
\begin{figure}[htbp]
  \begin{tabular}{cc}
  \begin{minipage}{0.4\hsize}
  \begin{center}
 \includegraphics[width=5cm, bb=0 0 200 200]{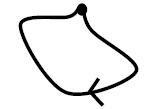}
\caption{Homoclinic orbit}
\end{center}
\end{minipage}
 \begin{minipage}{0.4\hsize}
  \begin{center}
 \includegraphics[width=5cm, bb=0 0 300 300]{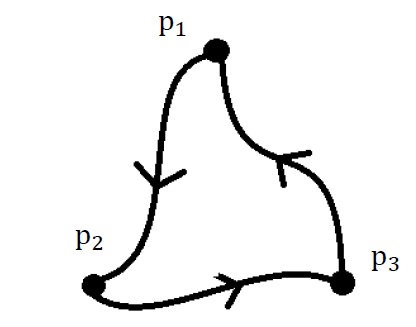}
\caption{Homoclinic cycle}
\end{center}
\end{minipage}
\end{tabular}
\end{figure}
 There are many results about the existence of homoclinic orbits in Hamiltonian systems. For obvious reasons homoclinic cycles is not appeared in dynamical systems of gradient-like vector fields of functions corresponding to the classical Lusternik-Schnirelmann category.
 \begin{thm}[M.Farber, \cite{far1} \cite{far2}]
 Let $\omega$ be a smooth closed 1-form on a closed smooth manifold $M$, and $\omega$ represents $\xi \in H^{1} (M; {\bf R})$ . If $\omega$ admits a gradient-like vector field $v$ with no homoclinic cycles, then $\omega$ has at least $\operatorname{cat} (M,\xi)$ distinct zeros.
 \end{thm}
 A slightly more informative statement was proved in follow theorem.
\begin{thm}[M.Farber, \cite{far1} \cite{far2}]\label{homoc}
Let $\omega$ be a smooth closed 1-form on a smooth closed manifold $M$ having fewer than $\operatorname{cat} (M, \xi)$ zeros, and $\omega$ represents cohomology class $\xi \in H^{1}(M, {\bf R})$. Then there exists an integer $N>0$ such that any gradient-like vector field $v$ for $\omega$ has a homoclinic cycle $\gamma_{1},..,\gamma_{n}$ satisfying
\begin{eqnarray}
\sum ^{n} _{i=1} \int _{\gamma _{i}} \omega \leq N.
\end{eqnarray}
\end{thm}
This shows that there always exist homoclinic cycles which cannot be destroyed while perturbing the gradient-like vector field. It is a new phenomenon, not occurring in Novikov theory.\par
These proves is in \cite{far1}, but it can be proved similarly as continuous versions in following sections.
\begin{cor}
The zeros of a smooth closed 1-form $\omega$ is all Morse type, then $\omega$ has at least $\operatorname{cat} (M,\xi)$ zeros. 
\end{cor}

 By the Kupka-Smale theorem \cite{ks} it is always possible to find a gradient-like vector field for $\omega$ such that any integral trajectory connecting two zeros comes out of a zero with higher Morse index and goes into a zero with lower Morse index. Such vector field has no homoclinic cycles.

\subsection{The zero points of a continuous closed 1-from}
 Let us consider the continuous version of the previous theorem. But we need define the zero point of a continuous closed 1-form. It is defined by using a gradient-like flows. Firstly let us recall the definition of continuous flow.
\begin{dfn}
The continuous function $\varphi : X \times {\bf R} \rightarrow X$ is said to be a continuous flow if it satisfies following:\\
(1) $\varphi (x,0) =x $, for any $x\in X$\\
(2) $\varphi (\varphi (x ,s ) , t ) = \varphi (x, s+t)$ for any $x\in X$ and for any $s,t \in {\bf R}$ 
\end{dfn} 
 $\varphi (x,t)$ is often denoted by $\varphi ^{t} (x)$ or $x \cdot t$.
\begin{dfn}
Let $X$ be a CW-complex, and let $\omega = \{ f_{U} \} _{U \in \mathcal{U} }$ be a continuous closed 1-form on $X$. The gradient-like flow of $\omega$ is defined as a continuous flow $\varphi ^{t}$ on $X$ such that on $U-(\operatorname{Fix}(\varphi ^{t} )\cap U )$, $f_{U}$ is a strictly Lyapnouv function, i.e. $f_{U}$ is strictly decreasing along a flow.
\end{dfn}
 We recall the definition of a zero-point of continuous closed 1-form.
\begin{dfn}
A point $p\in X$ is said to be zero point of $\omega$ if for any gradient-like flow of $\omega$, $p$ is a fixed point.
\end{dfn}
 Let us compare the definition with the zero points of smooth closed 1-form.
 \begin{lem}
Let $\omega$ be a smooth closed 1-form on a smooth Riemann compact manifold $M$ and let $p\in M$ is not zero point of $\omega$ in a sense of the smooth definition. Then $p$ is not a zero-point as a meaning of the continuous version.
 \end{lem}
\begin{proof}
 Since $M$ is compact manifold, there is a gradient-like flow for $\omega$. In this flow all fixed point is a zero point of $\omega$.
 \end{proof}
 
\begin{lem}
Let $f:X \rightarrow {\bf R}$ be a continuous function. If $p \in X$ is a maximum point of $f$, then $p$ is a zero point of $df$.
\end{lem}
 \begin{proof}
 Let $\varphi$ be a gradient-like flow of a continuous function $f$, and let $U \subset X$ be a open set include $p$ where $f(p)$ is maximal value of a function $f$ in $U$. Assume that $p\in X$ is not a fixed point of the flow $\varphi$.Then for small $\epsilon <0$,
$$
p \cdot \epsilon \in U, p \cdot \epsilon \neq p
$$
and, 
$$
f(p \cdot \epsilon) > f(p).
$$
This is contradict.
\end{proof}

The case that $p\in X$ is a minimum point is proved same as a maximum version. 
\begin{lem}
 Let $M^{n}$ be a smooth manifold and let $f: M \rightarrow {\bf R}$ be a morse function. Let $p\in M$ be a morse critcal point of $f$, then $p$ is a zero point of $df$ as continuous closed 1-form.
\end{lem}
\begin{proof}
 The case that the Morse index equals 0 or n, is proved in previous lemma.
 Assume that the Morse index of zero-point $p$ is $0<r<n$, and $f(p) =0$. In local coordination near $p\in M$ ,$f$ can be represented by
\begin{eqnarray}
f(x) = -x^{2}_{1} - \cdots - x^{2} _{r} + x^{2}_{r+1} + \cdots + x^{2}_{n}
\end{eqnarray}  
\begin{figure}
\begin{center}
\includegraphics[width=7cm,bb=0 0 400 400]{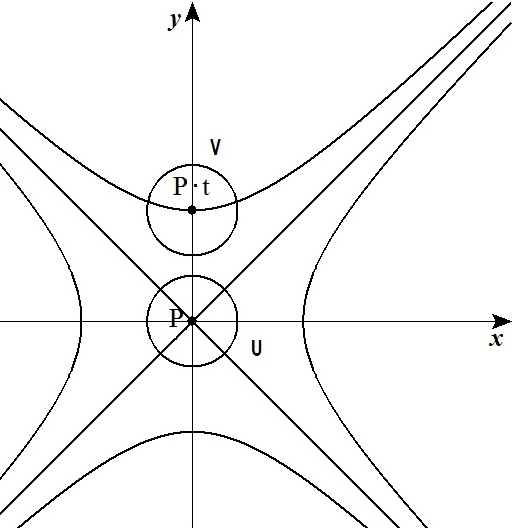}
\caption{$x=x_{1},...,x_{r}$, and $y=x_{r+1},...,x_{n}$}
\end{center}
\end{figure} 
Let $\varphi ^{t}$ be a gradient-like flow of $df$ as a meaning of continuous closed 1-form, and assume that $p\in M$ is not a fixed point of the flow $\varphi ^{t}$. Let $\epsilon >0$ be enough small. 
Let $V$ be a small neighborhood of $p \cdot \epsilon$ where $V \cap f^{-1} (0) \neq \phi$, and let $U$ be a small neighborhood of $p$ where $U \cdot \epsilon \subset V$. 
Let $\delta > 0$ be enough small, and $S= \{ x^{2}_{1} + \cdots + x^{2} _{r} =\delta ,x^{2}_{r+1} + \cdots + x^{2} _{n} =\delta \} $, where $S \subset f^{-1} (0), S \subset U$. 
Then $S \simeq S^{r-1}  \times S ^{n-r-1}$. By the assumption, $S \cdot \epsilon \subset V$. 
While $S$ flows until $\epsilon$, $S$ meets $\{ x_{r+1}= \cdots = x_{n} =0 \}$. But $\{ x_{r+1}= \cdots = x_{n} =0 \} \subset \cup _{c \leq 0 } f(c)$. It is contradict. Then $p \in M$ is a fixed point of $\varphi ^{t}$.
\end{proof}

\begin{example}
 It  proved that the Morse type zero point is a zero point in a sense of continuous version. But without the assumption of the Morse type there are examples that a zero point as meaning of a smooth 1-form is not a zero point  in a sense of a continuous 1-form.\par
 Let $X = {\bf R}$, and $f(x) = x^{3}$, for $x \in {\bf R}$. Let $\varphi (x,t) = x+ t$, then $\varphi$ is a gradient-like flow of a $df$ in a sense of a continuous closed 1-form. $x=0$ is a zero-point of a smooth closed 1-form $df$, but $x=0$ is not a fixed point of $\varphi$. \\
 \begin{figure}
 \centering
 \includegraphics[width=5cm, bb=0 0 400 500]{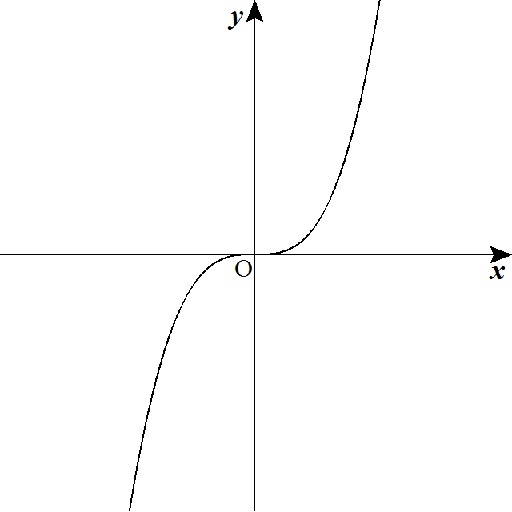}
\caption{$y = x^3$}
\end{figure}
\end{example}

\subsection{Estimate of the number of zeros of a continuous 1-form}
Later sections, we verify our main theorems.
 We defined the zero point of closed 1-form, so now let consider the continuous version of the theorem \ref{homoc}.
 \begin{thm}\label{main}
 Let $X$ be a finite CW-complex, and let $\omega$ be a continuous closed 1-form on $X$, that represents a cohomology class $\xi \in H^{1} (X ; {\bf R})$ Let $\phi$ be a gradient-like flow of $\omega$ and $\operatorname{Fix} (\varphi^{t}) $ is finite set ${p_{1},...,p_{k}} $, where integer $k>0$. If $k < \operatorname{cat}(X, \xi)$, then the flow $\varphi$ has a homoclinic cycle.
 \end{thm}
 Then by definition of zero points we conclude next theorem.
\begin{cor}\label{maincor}
(1)Let $X$ be a finite CW-complex, and let $\omega$ be a continuous closed 1-form, and $\xi =[ \omega ] \in H ^1 ( X; {\bf R})$.Suppose that $\omega$ has finite zero points $p_1,..,p_k$ . If $k < \operatorname{cat}(X, \xi)$, then any gradient-like flow that having fixed points only $p_1,..,p_k$ has a homoclinic cycle.\\
 (2)Let $\omega$ be a continuous closed 1-form on a finite CW-complex $X$, and $\omega$ represents $\xi \in H^{1} (X; {\bf R})$ . If $\omega$ admits a gradient-like flow with no homoclinic cycles, then $\omega$ has at least $\operatorname{cat} (X,\xi)$ distinct zeros.
 \end{cor}
 As we will show later,there exists the gradient flow in (1), that is fixed only zero-point of the continuous closed 1-form. So, if $k < \operatorname{cat}(X, \xi)$, then there always exist a gradient-like flow that has a homoclinic cycle. This proof is described following sections, and in the next section we prove the lemma for this theorem. 
 
 \subsection{Gradient Convex set} 
 To prove these theorems, we need a lemma below.
\begin{lem}\label{g-clem}
 Let $X$ be a CW-complex and let $f : X \rightarrow {\bf R} $ be a continuous function, and $\omega = df$. Let $\varphi ^{t} $ be a gradient-like flow of $\omega$, and assumed that $\operatorname{Fix} (\varphi )$ is a isolated set.Let $p \in X$ be a fixed point of $\varphi$. Let $W \subset M$ be a neighborhood of $Z$. Then there exist a open neighborhood of $p$, $U\subset W$, such that\\
 (1) For $x \in X$, $J_{x} = \{t \in {\bf R} | x \cdot t \in U \}$ is convex, that is, it is either empty or an interval.\\
 (2) if $A$ is the set of points $x \in X$ such that the interval $J_{x}$ is non-empty and bounded below, then the function 
 $$
 A \rightarrow {\bf R}, x \rightarrow \inf J_{x}
 $$
 is continuous.
 \end{lem}
 
 \begin{proof} Now we prove the (1) of Lemma.

  For simplicity, suppose that $\bar{W}$ is not contained other fixed point, and $f(p)=0$. For small $\epsilon >0$, let $\Psi _{\epsilon} : X \rightarrow {\bf R}$ be 
  $$
  \Psi _{\epsilon } (x ) = f (x \cdot (-\epsilon ) ) - f(x), 
  $$
  where $x\in X$.
 It is a continuous function, and $\Psi _{\epsilon} (p) = 0$. And let $h : V \rightarrow {\bf R}$ be 
 $$
 h(x) = \inf _{t \in I_{x}} \{ \Psi _{\epsilon} (x \cdot t) \}
 $$
 where $I_{x}=\{ t\in {\bf R} ; x \cdot t \in V\} $.
 \begin{clm}
 The function $h$ is a continuous.
 \end{clm}
 Let $ \{x_{i} \}$ be a sequence in $V$ converging $x \in V$ as $i \rightarrow \infty$.\par
 We suppose that $\varlimsup_{i \rightarrow \infty} h (x _{i} ) > h(x)$. If $h(x) \neq 0$, there exist $\tau \in {\bf R}$ such that $h(x) = \Psi _{\epsilon} (x \cdot \tau) $. Since the flow and $\Psi _{\epsilon}$ is both continuous, for small $\delta >0$ there exist a neighborhood of $p$ ,$U$, and real positive number, $\alpha >0$, such that  for any $y \in U$, and any $s \in [-\alpha, \alpha]$,
 \begin{eqnarray}
 |\Psi_{\epsilon} (y \cdot (\tau + s)) - h(x) | < \delta.
 \end{eqnarray}
 It contradict since for enough large any $i$, $x_{i}$ are in $U$. If $h(x) = 0$, for enough small $\delta >0$ there exist $\tau \in {\bf R}$ such that $\Psi _{\epsilon} ( x \cdot \tau)= \delta$. Then we conclude contradict similarly as above. So  $\varlimsup_{i \rightarrow \infty} h (x _{i} ) \leq h(x)$.\par
 Next we assume that $h(x) > \varliminf_{i \rightarrow \infty} h (x _{i} )$. If $\varliminf_{i \rightarrow \infty} h (x _{i} )$ is not 0, for enough small $\delta >0$ there exist a sequence $\tau _{i} \in {\bf R}$ such that for any $i$, $\Psi _{\epsilon} ( x_{i} \cdot \tau _{i}) < h(x)- \delta $. If necessary, we take a sub sequence of $\{ \tau _{i} \}$, and suppose that $\tau_{i} \rightarrow \tau (< \infty )$. Then
 \begin{eqnarray}
 \lim h(x _{i}) = \lim \Psi _{\epsilon} (x_{i} \cdot \tau _{i} ) = \Psi_{\epsilon} (x \cdot \tau)
 \end{eqnarray}
  It contradict. If $\varliminf_{i \rightarrow \infty} h (x _{i} )$ is $0$, we conclude a contradiction similarly to above. So,  $h(x) \leq \varliminf_{i \rightarrow \infty} h (x _{i} )$.\par
  Then $\varlimsup_{i \rightarrow \infty} h (x _{i} ) \leq h(x) \leq \varliminf_{i \rightarrow \infty} h (x _{i} ) $, hence the function $h(x)$ is continuous.\par
  
 Let $V, U, U'$ be open neighborhoods of $p$ such that $p \subset V \subset \bar{V} \subset U \subset \bar{U} \subset U \subset \bar{U'}  \subset W$. And let continuous $\mu , \eta : X \rightarrow {\bf R}$ such that
 \begin{eqnarray}
\mu = \left\{ \begin{split} &0 &on \, U \\
 &1 &on X-U' \end{split} \right.\\
\eta = \left\{ \begin{split} &1 & on \, U' \\
&0 &on X -W
\end{split} \right.
 \end{eqnarray}
 \begin{clm}
 There exist $\beta >0$ such that for $q \in \partial V, q\cdot t \in \partial U,$ ,where $t \in {\bf R}$,
\begin{eqnarray}
f (q ) - f(q \cdot t) > \beta
\end{eqnarray} 
 \end{clm}
 \begin{figure}[htbp]
  \begin{center}
 \includegraphics[width=5cm, bb=0 0 300 300]{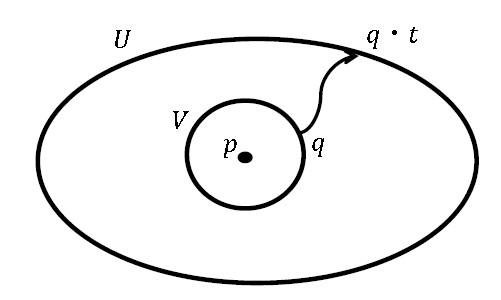}
\caption{Open set $U$ and $V$}
\end{center}
\end{figure}
 Suppose that there exist a sequence $\epsilon _{i} >0$ converging 0 as $i \rightarrow \infty$ and $q_{i} \in \partial V$ such that 
 \begin{eqnarray}
 f(q_{i}) - f(q_{i} \cdot t_{i}) < \epsilon,
 \end{eqnarray}
 where $q_{i} \in \partial U, t_{i} >0$. Losing no generality, we can assume that $q_{i} \rightarrow q \in \partial V$, as $i  \rightarrow \infty$. By assumption of $U,V$, $t_{i}$ does not converge to 0. Then enough small $\delta >0$, $\Psi _{\delta} (q_{i}) \rightarrow 0$, as $i \rightarrow \infty$. Hence $q\in \partial V$ is fixed point of the flow, but it contradict to the assumption that there is no fixed point in $W$, other $p$. So we conclude the above claim.\par
 Define two continuous functions $F_{+},F_{-} : X \rightarrow {\bf R}$:
 \begin{eqnarray}
 F_{\pm} := \pm f + \eta h + \mu 
 \end{eqnarray}
On $U-\{p \}$, $F_{+}$ is strictly decreasing along the flow, and $F_{-}$ is strictly increasing along the flow.\par
 Define $g : X \rightarrow {\bf R}$ as $g = \operatorname{max} \{ F_{+} , F_{-} \}$. It is a continuous function, and $g(p) =0$. Let $c>0$ be enough small for $U_{c} = \{ x \in X ; g(x) < c \} \subset V $, and $c < \beta$ Then $\bar{U}_{c} =\{x\in X ; g(x) \leq c  \}$, and $\partial \bar{U}_{c} = g^{-1} (c) $ We show that this $U_{c}$ satisfies lemma's conditions.\par
Let $x\in X$. The flow $x \cdot t$ enters in $\bar{U}_{c}$ at time $t=a$ if and only if 
\begin{eqnarray}
F_{+} (x \cdot a ) = c \geq F_{-} (x \cdot a). \label{flow}
\end{eqnarray}
The flow $x \cdot t$ leaves from $\bar{U}_{c}$ at time $t=b$ if and only if 
\begin{eqnarray}
F_{-} (x \cdot b ) = c \geq F_{+} (x \cdot b).
\end{eqnarray}
 Let show that the flow does not enter in $\bar{U}_{c}$ after leaving from $\bar{U}_{c}$. In $U$,since the function $F_{-}$ is strictly increasing along the flow, the flow does not satisfy \ref{flow}  twice. If the flow leaves from $U$ after leaving from $\bar{U}_{c}$, the value of $f$ increase more than $\beta$ by above claim. We assume $c< \beta$, hence this flow does not come back to $\bar{U} _{c}$. So, $\bar{U}_{c}$ satisfies lemma's conditions.
 
 Nextly we proof the (2).  The function $x \mapsto a_{x} $ is defined by $F_{+} (x \cdot a_{x}) =c$. Let $\delta >0$ be small. From the continuity of the flow and $F_{+}$, there exist a neighborhood of $x$ ,$U$ , such that for any $y\in U$
 \begin{eqnarray}
 F_{+} (y \cdot (a_{x} - \delta ) ) < c , F_{+} (y \cdot a_{x} + \delta ) > c.
 \end{eqnarray}
 From the intermediate value theorem there exist $a_{y} \in {\bf R}$ such that $a_{x} - \delta < a_{y} < a_{x} + \delta$, and $f(y\cdot a_{y}) = c$. So this function is continuous.
 \end{proof}
 
 \subsection{The proof of theorem \ref{main}}
 \begin{proof} of Theorem (\ref{main})\par
 Suppose that the flow $\varphi ^{t}$ does not have a homoclinic cycle. Fix a large integer $N>0$.  Let $U_{i}$ be a small closed contractible neighborhood of $p_{i}$, such that they are disjoint each other. 
 \begin{clm}
 There exist a closed neighborhood $V_{i} \subset U_{i}$ such that\\
 (a) $p_{i} \in V_{i} $, and $V_{i} \subset \operatorname{Int} U_{i}$\\
 (b) $V_{i}$ is a gradient convex set in following sense;\par
Let $\pi : \tilde{X}_{\xi} \rightarrow X$ is a covering corresponding to the kernel of period homomorphism, and let $f: \tilde{X}_{\xi} \rightarrow {\bf R}$ be a continuous function that $\pi ^{*} \omega = df$. Let $\tilde{ \phi} ^{t} $ be the lifted flow of $\phi ^{t}$. For any lift of $V_{i}$ $J_{x} = \{t\in {\bf R} ; x\cdot  t \in V_{i} \}$ is ether empty or it is a closed interval possibly half infinite or degenerated to a point, where $x\in \tilde{X}_{\xi}$.\\
(c) Let $\partial _{-} V_{i} $ denote the set of $p \in \partial V_{i}$ such that for any sufficiently small $\tau >0$, $p \cdot \tau \notin V_{i}$. Then we require that for any $p\in \partial _{-} V_{i}$ there exist no real number $t_{p} > 0$ such that $p \cdot t_{p} \in \operatorname{Int} V_{i},$ and $\int _{p} ^{p\cdot t_{p} } \omega \leq -N$, where the integral in calculated along the flow $\sigma _{p} : [0, t_{p} ] \rightarrow X, \sigma _{p} (t) = p\cdot t$.
\end{clm}
Firstly prove next claim.
\begin{clm}
There exist $b > 0$, such that for any $p \in \partial _{-} V_{i}$ if there exist $t>0$ such that $p \cdot t \in V_{j}$, then 
\begin{eqnarray}
\int _{p} ^{p \cdot t} \omega < -b
\end{eqnarray}
\end{clm}
 Suppose that there exist a sequence $\epsilon _{n} >0$ converging 0 as $i \rightarrow \infty$ and $q_{n} \in \partial V_{i}$ such that 
 \begin{eqnarray}
  \int _{q_{n}} ^{q_{n} \cdot t_{n}} \omega< \epsilon,
 \end{eqnarray}
where $t_{n}>0$, and $q_{n} \cdot t_{n} \in V_{j}$. Since any $V_{i}$ is sufficiently small, losing no generality we can assume that $t_{n} \rightarrow \tau >0$, and $q_{n} \rightarrow q\in \partial V_{i}$ as $n \rightarrow \infty$. Then for sufficiently small $\delta >0$, $\Psi _{\delta} (q _{n}) < \epsilon _{n}$, so $\Psi _{\delta} (q) =0$. It contradict to that $q$ is not a fixed point. \par

 For previous lemma, we can choose an open set $V_{i}$ satisfying (a) and (b) of the claim. Suppose that we can never achieve (c) by shrinking the open set $V_{i}$, satisfing conditions (a) and (b). Then, there exist a seqence $p_{i,n} \in \partial _{-} V_{i}$, $t_{i,n} $ and $s_{i,n}$ for $i=1,2,\cdots $, such that\\
 (1) $p_{i,n} \cdot [s_{i,n} , 0] \subset V_{i} $,and $p_{i,n} \cdot s_{i,n} \rightarrow p_{i} $ for $n \rightarrow \infty$\\
 (2) $p_{i,n} \cdot t_{i,n} \rightarrow p_{i} $ as $n \rightarrow \infty$.\\
 (3)$\int_{p_{i,n} } ^{p_{i,n} \cdot t_{i,n}} \omega \leq -N$\par
 We assume that $p_{i,n} \rightarrow q_{i} \in \partial _{-} V_{i} $. Let $\pi : \tilde{X}_{\xi} \rightarrow X$ be the covering of the period homomorphism of $\xi$, and let $f : \tilde{X} _{\xi} \rightarrow {\bf R}$ be a function such that $\pi ^{*} \omega = df $. And let fix the one of lifted $p_{i} $ and $V_{i}$ identify $V_{i}$ with the lifted one.\par
 Next we want to show $\lim _{t \rightarrow - \infty} q_{i} \cdot t = p_{i}$. If for sufficiently small $\epsilon >0$, there exist $s<0$ such that $f(q_{i} \cdot s ) >f(p_{i} ) + \epsilon $, then for sufficiently large $n>0$, $f(p_{i,n}) > f(p_{i}) + \epsilon$ because the set $\{ (x,s) \in \tilde{X}_{\xi} \times {\bf R} ; f(x\cdot s ) > f(p_{i}) + \epsilon \}$ is open. This contradict to the condition (1), so $f(q_{i} \cdot t) \rightarrow f(p_{i} ) $ as $t \rightarrow - \infty$. Hence we can conclude $q_{i} \rightarrow p_{i}$.
 Next let check that $\lim_{t \rightarrow \infty} q_{i} \cdot t$. There are 2 patterns. The trajectory $q_{i} \cdot t$ in $\tilde{X}_{\xi} $ for large $t$ may either reach the neighborhood $g V_{i}$, where $g$ is a element of translate group, or it may be "caught" by some other fixed point on the way.\par
 In the first case, $\lim_{t \rightarrow \infty}q_{i} \cdot t = g p_{i}$. In otherwise it contradict to the condition (2) similarly to above. In this case the flow has a homoclinic cycle.\par
In the second case, $q_{i} \cdot t \rightarrow p_{j} \neq p_{i}$ on $X$. Suppose that $q_{i} \cdot t \rightarrow h p_{j}$ on $\tilde{X}_{\xi}$ where $h$ is a element of the translate group. From previous claim,
$$
f(g \cdot p_{i} ) + b < f(h p_{j}), f(h p_{j}) + b < f(p_{i}).
$$
 For sufficient large $n$, the trajectory $p_{i,n} \cdot t$ enters in $h V_{j}$ at some point $p'_{i,n} \in \partial (hV_{j}) $, and leave from $h V_{j}$ at some point $p'' _{i,n} \in \partial V_{j}$. Let $\tau _{i,n} \in {\bf R}$ be a positive number that $p' _{i,n} \dot \tau _{i,n} = p'' _{i,n}$. We suppose $p'_{i,n}, p''_{i,n},\tau_{i,n}$, and 
 \begin{eqnarray}
 \lim_{n \rightarrow \infty} p'_{i,n} = q'_{i} \in \partial (h V_{j}) \\
 \lim_{n \rightarrow \infty} p''_{i,n} = q''_{i} \in \partial _{-} (h V_{j}). 
 \end{eqnarray}
 Then $\tau_{i,n \rightarrow \infty}$ as $n \rightarrow \infty$, and similarly to above case we can conclude $\lim _{t \rightarrow - \infty} q'_{i} \cdot t = p_{i}, \lim_{t \rightarrow \infty} q'_{i,n} \cdot t = hp_{j} = \lim_{t\rightarrow -\infty} q''_{i} \cdot t $. And $\lim_{t \rightarrow \infty} q''_{i,n} \rightarrow gp_{i}$ or $h_{1} p_{j_{1}}$ in $\tilde{X}_{\xi} $ where $h_{1}$ is an element of the translate group.\par
 Continuing by induction we obtain that downstairs a homoclinic cycle starting and ending at $p_{i}$.The number of steps in the above process is finite at most $[N / b]$.\par
 This proves the existence of the disks $V_{i}$ with properties (a), (b), (c) assuming the flow has no homoclinic cycles.\par
\begin{figure}[htbp]
  \begin{center}
 \includegraphics[width=5cm, bb=0 0 200 300]{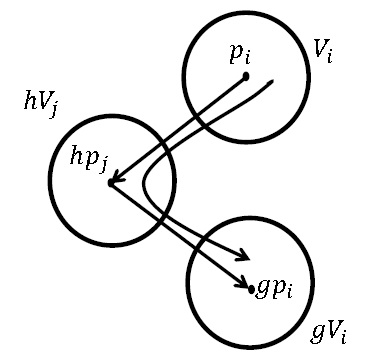}
\caption{Trajectories in $\tilde{X}_{\xi}$}
\end{center}
\end{figure}
 Next we will construct an open cover 
 \begin{eqnarray}
 F \cup F_{1} \cup \cdots \cup F_{k} = X . 
 \end{eqnarray}
 The set $F$ is defined as the set of all points $p\in X$ such that there exists a positive number $t_{p}$, so that, the flow $\sigma _{p} : [0, t_{p}] \rightarrow X$, where $\sigma_{p} (t) = p\cdot t$, satisfies $\int _{\sigma _{p}} \omega = -N$. Then since the flow is continuous $F$ is open and the map $p \rightarrow t_{p} $ is continuous.\par
 Next the set $F_{j}$ is defined as the set of all points $p \in X$ such that there exist $t_{p} >0$, so that $p \cdot t_{p} \in \operatorname{Int} V_{j},$ and 
 \begin{eqnarray}
 \int_{p} ^{t_{p}} \omega > -N. 
 \end{eqnarray}
 For the continuity of the flow $F_{j}$ is open. \par
We will show $F_{j}$ is contractible. For $p\in X$ let $J_{p} := \{ t \leq 0 ; p \cdot t \in V_{j} \}$. For the assumption of $V_{j}$, $J_{p}$ is a disjoint union of closed interval. Let $[\alpha _{p} , \beta _{p}] \subset J_{p}$ be the first interval. If $[\alpha_{p} , \beta _{p}]$ is a point, $p$ does not belong to $F_{j}$, similarly if $p$ is in $\partial _{-} V_{j}$ $p$ is not in $F_{j}$. When $p \in F_{j}$, and $p$ is not in $\operatorname{Int} V_{j}$, $p \cdot t$ is in $\operatorname{Int} V_{j} $ for $\alpha _{p} < t < \beta_{p}$, and $\int_{p} ^{p \cdot \alpha _{p}} \omega > -N$. $\phi_{j} : F_{j} \rightarrow {\bf R}$ is defined as the function such that $\phi _{j} (p) =0$ when $p$ is in $\operatorname{Int} V_{j}$, and $\phi_{j} (p) = \alpha$ when $p$ is in $F_{j} - \operatorname{Int} V_{j}$. For the property of $V_{j}$ $\phi_{j} $ is continuous. Then the homotopy $h_{\tau} : F_{j} \rightarrow X$ is defined as the homotopy such that $h_{\tau} (p) = p (\tau \phi_{j} (p))$, for $\tau \in [0,1]$. Then $h_{0} $ is an inclusion, and $h_{1}$ maps $F_{j}$ into the disk $V_{j}$. So, $F_{j}$ is contractible.\par
 From definitions $X = F \cup F_{1} \cup \cdots \cup F_{k}$. This open cover satisfies the definition of the category with respect to the cohomology class. So, $\operatorname{cat} (X, \xi) > k$.
 \end{proof}
 
 \subsection{Gradient-like flow that fixes zero-points}
 \begin{prop}
 $X$ is finite CW-complex, and $\omega$ is a continuous closed 1-form on $X$ which have finite zero-points $p_1,...,p_k$. There exists a gradient-like flow of the continuous closed 1-form such that fixes only $p_1,...,p_k$. 
 \end{prop}
 
 \begin{proof}
Let $p$ be a zero-point of $\omega$. Choose two gradient-like flows of this continuous closed 1-form, $\phi$ and $\varphi$ that $\phi$ fixes $p$, and $\varphi$ does not fixes $p$. We choose $U_c$ for small enough $c>0$ in a proof of Lemma \ref{g-clem} of gradient flow $\phi$.  So, the flow $\phi$ does not enter in $\bar{U}_c$ after leaving from $\bar{U}_c$. We can assume that the flow $\phi$ enter in $\bar{U_c}$ at $\partial U_c^-$, and leave from $\bar{U}_c$ at $\partial U_c^+$. $\omega$ can be represented by $df$ around $p$, where $f$ is a cotinuous funstion near $p$, and $f(p)=0$. Then $f\geq0$ on $\partial U_c^+$, and $f\leq0$ on $\partial U_c^-$.

 Let show the other flow $\varphi$ carries $\partial U_c^-$ to $\partial U_c^+$. Let $c_0=c$ be fixed, and assume $c$ is small enough. We assume that  $\varphi$ enters in $\bar{U}_{c_i}$ at $\partial U_{c_i}^+$, and leave from $\bar{U}_{c_i}$ at $\partial U_{c_i}^+$ as $c_i\rightarrow 0$.So we assume the trajectory of the flow $\varphi$ meets $\partial U_c^+$ at $q_i$, carries $q_i$ to $\partial U_{c_i}^+$ and leave from $U_{c_i}^+$ and $U_c^+$. The trajectory leave from $\partial U_{c_i}^+$ at $\varphi_{\delta_i} (q_i)$, and leave from $\partial U_c^+$ at $\varphi_{\delta'_i} (q_i)$. For shrinking $U_{c_i}$ as $c_i \rightarrow 0$, 
 \begin{eqnarray}
 f(q_i) \rightarrow 0.
 \end{eqnarray}
  We denote the limit of $q_i$ $q_{\infty}$ as $i\rightarrow \infty$. Under assumption $\varphi$ carries $p$ to $\partial U_c^+$ at $\varphi_{\delta'}(p)$. As $i \rightarrow \infty$, $\varphi_{\delta_i} (q_i)$ tends to the point $p$. So, $\varphi_{\delta'} (q_i)$ tends to $\varphi_{\delta'} (p)$ for continuity of $\varphi$. So, the limit of this trajectory pass through $q_\infty$, $p$ and $\varphi_{\delta'} (p)$. But 
  \begin{eqnarray}
  f(q_\infty) = f(p)=0,
  \end{eqnarray}
 it contradict that $p$ is not fixed point of the gradient flow $\varphi$. And so on the $\partial U_c^-$. So, $\varphi$ carries $\partial U_c^-$ to $\partial U_c^+$.
 
So, we can construct a new gradient flow from $\phi$ by replacing the flow $\varphi$ in a neighborhood of zero-point $p$, $U_c$. Zero-points are finite, then this step is stopped in finite times. So, we can construct a gradient-flow of $\omega$ that fixes only zero-points of $\omega$.
\begin{figure}[htbp]
  \begin{center}
 \includegraphics[width=5cm, bb=0 0 300 300]{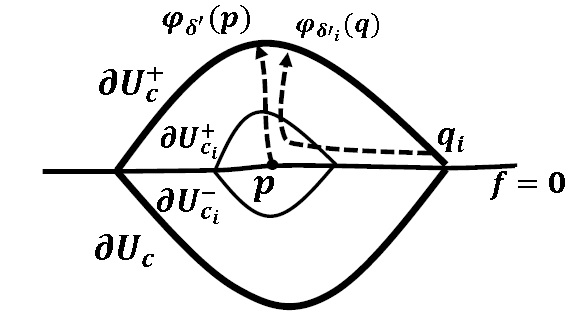}
\caption{The set $\partial U_c^-$ and $\partial U_c^+$}
\end{center}
\end{figure}
 \end{proof}
 
 For this proposition, we can take the opitimum gradient-like flow of a continuous closed 1-form. So, for (1) of Cororaly \ref{maincor} if the number of zero-point is less than $\operatorname{cat}(X, \xi)$, then there exists a gradient-like flow of the continuous closed 1-form that has a homoclinic cycle.
 
\section{Examples}
 In this section we denote examples of continuous closed 1-forms and Lusternik Schnirelman categories with respect to the cohomology class. Let $T_1$ be a closed surface whose genus is 2, and let $T_2$ be a closed surface constructed by shrinking a hole from $T_1$. And we define a closed 1-form$d \theta _1$ on $T_1$ by the angle that axis to a hole, and so on $d\theta_2$ on $T_2$. $T_1$ is a smooth manifold, so $d\theta_1$ can be realized as a smooth closed 1-form, but $T_2$ is not a smooth manifold but a finite CW-complex so $d\theta _2$ is not a smooth closed 1-form, but a continuous closed 1-form. They are described in figure \ref{fig:one}, \ref{fig:two}. So, Farber's Lusternik-Schnirelmann theory can be applied to $d\theta_1$, but not to $d\theta_2$. But we applied our continuous version of Lustenik-Schnirelmann theory to both closed 1-forms. 

The number of zero-point of $d\theta_1$ is 2, and one of $d\theta _2$ is 1. And Lustenik-Schnirelman category with respect to cohomology class is established following formulas,
\begin{eqnarray}
\operatorname{cat} (T_1, [d \theta _1] ) =1\\ 
\operatorname{cat} (T_2, [d\theta _2] ) =1.
\end{eqnarray}
\begin{figure}[hbt]
 \begin{minipage}{0.4\hsize}
  \begin{center}
   \includegraphics[width=70mm,bb=0 0 500 500]{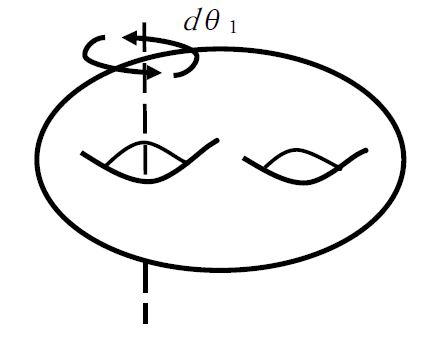}
  \end{center}
  \caption{$T_1$}
  \label{fig:one}
 \end{minipage}
 \begin{minipage}{0.4\hsize}
  \begin{center}
   \includegraphics[width=70mm,bb=0 0 600 600]{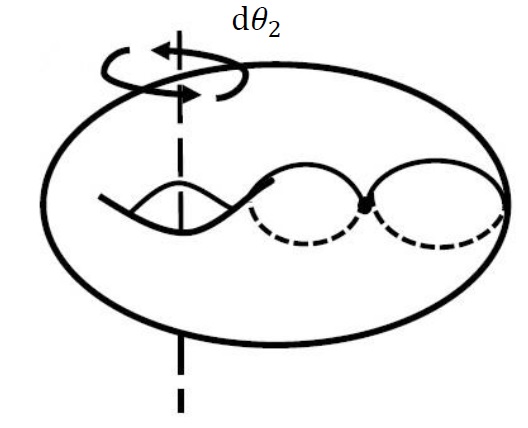}
  \end{center}
  \caption{$T_2$}
  \label{fig:two}
 \end{minipage}
\end{figure}

\begin{remark}
 The examples that the number of closed 1-form is less than Lusternik-Schnirelman category, that is any gradient-like flow of closed 1-form has homoclinic cycles, is showed by M.Farber in \cite{far1}. This example can be constructed by a smooth manifold. Its 1-form can be regarded as continuous closed 1-form, and it is also an exmple as a continuous Lusternik-Schnirelmann category.
 \end{remark}

\end{document}